\documentclass[12pt]{amsart}
\usepackage{nccmath}

\usepackage{amsmath,amsthm,amssymb}
\usepackage{times}
\usepackage{enumerate}

\newtheorem{theorem}{Theorem}[section]

\newtheorem{lemma}[theorem]{Lemma}

\newcommand{\be}{\begin{equation}}
\newcommand{\ee}{\end{equation}}
\newcommand{\ol}{\overline}
\newcommand{\ul}{\underline}
\newcommand{\tu}{\tilde{u}}
\newcommand{\ue}{u^\epsilon}

\newcommand{\lt}{\left}
\newcommand{\rt}{\right}

\newcommand{\goto}{\rightarrow}

\newcommand{\e}{\epsilon}
\newcommand{\s}{\sigma}

\newcommand{\RNum}[1]{\uppercase\expandafter{\romannumeral #1\relax}}

\theoremstyle{definition}
\newtheorem{defin}[theorem]{Definition}
\newtheorem{remark}[theorem]{Remark}

\newtheorem{question}[theorem]{Question}

\numberwithin{equation}{section}

\begin{document}
\setlength{\baselineskip}{1.2\baselineskip}

\title[Non-convexity of level sets]
{Non-convexity of level sets for $k$-Hessian equations in convex ring}

\author{Zhizhang Wang}
\address{School of Mathematical Science, Fudan University, Shanghai, China}
\email{zzwang@fudan.edu.cn}
\author{Ling Xiao}
\address{Department of Mathematics, University of Connecticut, Storrs, Connecticut 06269}
\email{ling.2.xiao@uconn.edu}
\thanks{The first author is supported by NSFC Grants No.12141105.}

\begin{abstract}
In this paper we construct explicit examples that show the sublevel sets of the solution of a $k$-Hessian equation defined
on a convex ring do not have to be convex.
\end{abstract}

\maketitle

\section{Introduction}
\label{sec0}

A domain $\Omega\subset\mathbb R^n$ is called a {\it  convex ring} if
$\Omega=\Omega_1\setminus\overline\Omega_2,$
where $\Omega_1$ and $\Omega_2$ are two bounded convex domains in $\mathbb R^n$ such that
$\overline\Omega_2\subset\Omega_1.$ In this paper, we will concern with the non-convexity of sublevel sets for the solution $u$ to
the $k$-Hessian equation defined on a convex ring
\be\label{eq0.0}
\left\{
\begin{aligned}
F_k(D^2u)&=C>0\,\, &\mbox{in $\Omega=\Omega_1\setminus\ol\Omega_2,$}\\
u&=0\,\,&\mbox{on $\partial\Omega_1,$}\\
u&=-1\,\,&\mbox{on $\partial\Omega_2.$}
\end{aligned}
\right.
\ee
Here, $0<C\leq C_0$, $C_0=C_0(\Omega)$ is a positive constant depending on the domain $\Omega,$  and the $k$-Hessian operator $F_k$ is defined by
\[F_k[u]=\sigma_k(D^2u)=\sigma_k(\lambda(D^2 u)),\]
where $\lambda=(\lambda_1, \cdots, \lambda_n)$ denotes the eigenvalues of the Hessian matrix of $D^2u,$ and $\sigma_k$ is the $k$-th elementary symmetric function on
$\mathbb R^n.$

It is believed that solutions of boundary value problems for elliptic equations often inherit important geometric properties of the domain with the influence of the structures of the corresponding equations. Therefore, a typical question to ask is
\begin{question}
\label{question1}
If $u$ is a non-positive solution to equation \eqref{eq0.0}, is it true that the sublevel set of $u,$
that is, $\{x\mid u\leq c\}$ are all convex?
\end{question}

For $k=n$, since the admissible solution for \eqref{eq0.0} is strictly convex, it is trivial that the above question has a positive answer. However, for $k<n$, the admissible solutions are strictly $k$-convex, the above question becomes very interesting.

In the literature, for general nonlinear equations \be\label{eq0.1*}F(D^2u, Du, u,x)=0,\ee Question \ref{question1} has been studied intensively.

The convexity of level-sets of solutions for harmonic equations was first studied by Gabriel \cite{Gab}. Lewis \cite{Lew} extended Gabriel's results to $p$-harmonic functions. Caffarelli and Spruck \cite{CS} treated this problem for inhomogeneous Laplace equations that are in connection with free boundary problems. Kawhol \cite{Kaw} proposed an approach of using quasi-concave envelop to study the convexity of the level-set of solutions to certain equations. Colesanti-Salani \cite{CoSa} carried out this approach for a class of elliptic equations. The technique was extended by Cuoghi-Salani \cite{CuSa} and Longinetti-Salani \cite{LoSa} for general equations \eqref{eq0.1*} defined in convex ring under various structure conditions on $F$. In \cite{BLS}, Bianchini-Longinetti-Salani furthered this technique and proved the convexity of level-sets for solutions to \eqref{eq0.1*} with milder structure conditions on $F$.

A different approach to study the convexity of level-sets of solutions is to use the microscopic convexity principle. In particular, the constant rank theorem for the second fundamental forms of level sets of solutions to certain type of quasilinear equations was established by Korevaar \cite{Kor}, see also Xu \cite{Xu}. For $p$-harmonic function, the corresponding constant rank theorem is proved by Ma-Ou-Zhang \cite{MOZ}. More specifically, they gave a positive lower bound for the Gauss curvature of the convex level set of $p$-harmonic functions that depends on the Gauss curvature of $\partial\Omega$. In \cite{CMY}, Chang-Ma-Yang proved a similar result for inhomogeneous Laplace equations. Later, Bian-Guan-Ma-Xu \cite{BGMX} and  Guan-Xu \cite{GX} gave a lower bound for the second fundamental form of the level surface of solutions to \eqref{eq0.1*} in convex ring for a large class of elliptic operators $F$ by establishing the constant rank theorem.

It is equally interesting to find examples of solutions of PDEs in convex ring that have non-convex level-sets. When $n=2,$ Monneau-Shahgholian \cite{MoSh} showed that there exists a solution $u$ to the equation $\Delta u=f(u)$ defined in a convex ring $\Omega\subset\mathbb R^2$ with level-sets not all convex. Later, Hamel-Nadirashvili-Sire \cite{HNS16} constructed examples for the same equation but in arbitrary dimensions. Moreover, the conditions that the function $f$ needs to satisfy in \cite{HNS16} are more relaxed than in \cite{MoSh}.  However, to the best of the authors knowledge, it seems that there is no such examples for fully nonlinear equations in literature. In this paper, we extend the result of \cite{HNS16} to $k$-Hessian equations \eqref{eq0.0} and construct such examples.

Before we state our main result, we need the following definition.
\begin{defin}
\label{def1.0}
For a domain $\Omega\subset\mathbb R^{n}$, a function $u\in C^2(\Omega)$
is called strictly $k$-convex if the
eigenvalues $\lambda(D^2u)=(\lambda_1, \cdots, \lambda_n)$ of the hessian $D^2u$ is in $\Gamma_k$ for all $x\in\Omega,$ where
$\Gamma_k$ is the Garding's cone
\[\Gamma_k=\{\lambda\in\mathbb R^n\mid \sigma_j(\lambda)>0, j = 1,\cdots , k\}.\]
\end{defin}

\label{sub1.1}
Let us consider the $k$-Hessian equations
in a convex ring $\Omega=\Omega_1\setminus\overline\Omega_2,$
\be\label{k-hessian}
\left\{
\begin{aligned}
\sigma_k(D^2 u)&=1\,\, &\mbox{in $\Omega=\Omega_1\setminus\ol\Omega_2,$}\\
u&=0\,\,&\mbox{on $\partial\Omega_1,$}\\
u&=-M\,\,&\mbox{on $\partial\Omega_2,$}
\end{aligned}
\right.
\ee
where $M>0$ is a positive real number. For any classical solution $u$ of \eqref{k-hessian}, we define the continuous function
$\tu\in C(\ol\Omega_1)$ by
\[
\tu=\left\{
\begin{aligned}
& u(x)\,\,&\mbox{if $x\in\ol\Omega,$}\\
&-M\,\,&\mbox{if $x\in\Omega_2.$}
\end{aligned}
\right.
\]
We say that $u$ is {\it quasiconvex} in $\Omega$ if $\tu$ is so in $\Omega_1,$ that is,
if the sublevel sets
\[\bar\Omega^\lambda:=\{x\in\Omega_1: \tu(x)\leq\lambda\}\]
are convex for all $\lambda\leq 0.$

Our main result is the following.
\begin{theorem}
\label{thm1}
Let $\Omega_1$ be any smooth bounded convex domain in $\mathbb R^n.$ Then there exits a constant $M_1=M_1(\Omega_1)>0$
such that for all $M>M_1,$ there are some smooth convex rings $\Omega=\Omega_1\setminus\ol\Omega_2$ for which problem \eqref{k-hessian} with $n\geq 2k$
has a unique solution $u$ that is NOT quasiconvex.
\end{theorem}

\begin{remark}
 Our technique can be extended to the case when the right hand side of \eqref{k-hessian} is not a constant. Here, we use the constant function to keep the proof clean.  At this moment, we do not know if Theorem \ref{thm1} is true for $2k>n$. We think it may be an interesting problem to investigate.
\end{remark}

The problem \eqref{k-hessian} can be rewritten to \eqref{eq0.0}. In fact, suppose $u$ is the solution of \eqref{k-hessian}, let $u_1=u/M$ then $u_1$ satisfies \eqref{eq0.0} with
$C=1/(M)^k$. We can also rescale the domain $\Omega$ to keep the right hand side of the equation to be 1. In particular, we let
$$u_2(y)=\frac{u(\sqrt{M}y)}{M}.$$ Then $u_2$ satisfies \eqref{eq0.0} with $C=1$ while the domain becomes $\Omega/\sqrt{M}.$

We follow the frame work of \cite{HNS16} to construct counterexamples for $k$-Hessian equations. Our biggest obstacle is that the pure interior estimates are not true for $k$-Hessian equations when $k\geq 3$ (see \cite{Pog}). Therefore, we adapt the ideas of Hessian measure (see \cite{TW99, TW02}) to overcome this difficulty, which is the novelty of this paper.

\section{solvability of the Dirichlet problem }
\label{dirichlet problem}
In this section, we will consider the solvability of the following Dirichlet problem
\be\label{dp.1}
\left\{
\begin{aligned}
\sigma_k(D^2 u)&=1\,\,&\mbox{in $\Omega^\e:=\Omega_1\setminus\ol B_\e(0),$}\\
u&=0\,\,&\mbox{on $\partial\Omega_1,$}\\
u&=-M\,\,&\mbox{on $\partial B_\e(0).$}
\end{aligned}
\right.
\ee
Here $B_\e(0)\subset\Omega_1$ is a ball centered at the origin with radius $\e$ and $M\geq M_1(\Omega_1)>0$ is a positive real number. In this paper, we will always assume $\e\in(0, \e_0)$ to be a small constant.

\subsection{Choice of $M_1$ and $C^0$ estimates}
\label{dpsub1}
In this subsection, we will discuss the choice of $M_1$ as well as the $C^0$ estimates of the solution to \eqref{dp.1}, which we denote by $u^\e.$

First, let us consider the following Dirichlet problem
\be\label{dp.2}
\left\{
\begin{aligned}
\sigma_k(D^2 u)&=1\,\,&\mbox{in $\Omega_1$,}\\
u&=0\,\,&\mbox{on $\partial\Omega_1.$}
\end{aligned}
\right.
\ee
By Theorem 1 in \cite{CNS3}, we know that there exists a unique strictly $k$-convex solution $\psi\in C^\infty(\ol\Omega_1)$ to \eqref{dp.2}. In view of the standard maximum principle,
we have $\psi<0$ in $\Omega_1.$ We will choose $M_1>-\min\limits_{\ol\Omega_1}\psi$ such that $\psi$ is a supersolution of \eqref{dp.1}.

Second, it is easy to verify that
\be\label{lower barrier}
\ul u^\e_{M}=\frac{|x|^2}{2C_n^k}-M-\frac{\e^2}{2C_n^k}\ee
is a rotationally symmetric solution to the equation $\sigma_k(D^2 u)=1.$
Moreover, $\ul u^\e_{M}$ also satisfies the inside boundary condition of \eqref{dp.1}, that is, $\ul u^\e_M=-M$ on $\partial B_\e(0).$
We will choose $M_1=M_1(\Omega_1)>0$ such that $\ul u^\e_M$ is a subsolution of \eqref{dp.1}. In other words, we will choose $M_1>0$ large enough such that
$$\ul u^\e_{M_1}=\frac{|x|^2}{2C_n^k}-M_1-\frac{\e^2}{2C_n^k}\leq0\,\,\mbox{on $\partial\Omega_1.$}$$

Finally, we conclude that in this paper, $M_1>0$ is a positive constant only depending on $\Omega_1.$ In particular, $M_1$ is chosen such that $\psi$ and $\ul u^\e_{M}$ are the supersolution and subsolution to \eqref{dp.1} respectively. In the rest of this paper, for our convenience we will write $\ul u^\e$ instead of $\ul u^\e_M.$

Combining the above discussions with maximum principle we obtain the following $C^0$ estimate for the solution of \eqref{dp.1} directly.

\begin{lemma}
\label{dp-c0-lem}
Let $\ue$ be the solution of \eqref{dp.1}, then $\ue$ satisfies
\[\ul u^\e<\ue<\psi,\,\,\mbox{in $\Omega^\e.$}\]
\end{lemma}

\subsection{$C^1$ estimates of $\ue$}
\label{dpsub2}
In order to obtain the $C^1$ estimates of $\ue,$ we need to divide the discussion into three cases, i.e., $\frac{n}{k}>2,$
$\frac{n}{k}=2,$ and $\frac{n}{k}<2.$

Case 1. when $\frac{n}{k}>2,$ let
\[\phi=-C|x|^{2-n/k}+C\e^{2-n/k}-M\] for some
$C=C_0\e^{n/k-2}>0$ such that $\phi|_{\partial\Omega_1}\geq 0.$ Then we can check that $\phi$ is a supersolution of \eqref{dp.1} satisfying $\sigma_k(D^2\phi)=0$ and $\phi|_{\partial B_\e(0)}=-M.$ Moreover, it's easy to see that here $C_0$ is a positive constant only depending on $\Omega_1.$

Case 2.  when $\frac{n}{k}=2,$ let
\[\phi=C\log|x|-C\log\e-M\] for some
$C=\frac{C_0}{|\log\e|}>0$ such that $\phi|_{\partial\Omega_1}\geq 0.$ Then we can check that $\phi$ is a supersolution of \eqref{dp.1} satisfying $\sigma_k(D^2\phi)=0$ and $\phi|_{\partial B_\e(0)}=-M.$ Moreover, it's easy to see that here $C_0$ is a positive constant only depending on $\Omega_1.$

Case 3. when $\frac{n}{k}<2,$ let
\[\phi=C_0|x|^{2-n/k}-C_0\e^{2-n/k}-M\] for some
$C_0>0$ such that $\phi|_{\partial\Omega_1}\geq 0.$ Then we can check that $\phi$ is a supersolution of \eqref{dp.1} satisfying $\sigma_k(D^2\phi)=0$ and $\phi|_{\partial B_\e(0)}=-M.$ Moreover, it's easy to see that here $C_0$ is a positive constant only depending on $\Omega_1.$

In the following, when there is no confusion, we will not differentiate the three cases and just use $\phi$ to denote the supersolution of \eqref{dp.1} that
satisfies $$\sigma_k(D^2\phi)=0 \,\,\text{and}\,\, \phi|_{\partial B_\e(0)}=-M.$$
Combining with the subsolution $\ul u^\e$ constructed in Subsection \ref{dpsub1} and the maximum principle we get
\begin{lemma}{\textbf{[$C^1$- bounds on $\partial B_\e(0)$]}}
\label{dp-c1inside-lem}
Let $\ue$ be the solution of \eqref{dp.1}, then on $\partial B_\e(0),$ $\ue$ satisfies
\[\frac{\partial\ul u^\e}{\partial\nu}<\frac{\partial\ue}{\partial\nu}<\frac{\partial\phi}{\partial\nu}.\]
Here $\nu$ is the inward unit normal of $\partial\Omega^\e,$ i.e., $\nu$ points into $\Omega_1\setminus\ol B_\e(0).$
\end{lemma}
\begin{remark}
\label{dprmk1}
We should keep in mind that on $\partial B_\e,$ when $\frac{n}{k}>2,$ $|D\ue|<C_0\lt(\frac{n}{k}-2\rt)\e^{-1};$
when $\frac{n}{k}=2,$ $|D\ue|<C_0\e^{-1}|\log\e|^{-1};$ and when $\frac{n}{k}<2,$ $|D\ue|<C_0\lt(2-\frac{n}{k}\rt)\e^{1-n/k}.$
Here $C_0>0$ is a positive constant only depending on $\Omega_1.$
\end{remark}

Next, we want to obtain the $C^1$ bound on $\partial\Omega_1.$ In Subsection \ref{dpsub1} we already constructed a supersolution $\psi$ of \eqref{dp.1}
that satisfies $\psi<0$ in $\Omega_1$ and $\psi=0$ on $\partial\Omega_1.$ It is easy to see that for all $0<\e<\e_0$, we have $\psi|_{\partial B_\e(0)}<-c_0.$ Here $c_0>0$ only depends on
$\e_0.$ Now, let $C=C(M, \e_0)>1$ be a large constant such that $c_0C\geq M.$ Then $C\psi$ is a subsolution of \eqref{dp.1} satisfying $C\psi=0$ on $\partial\Omega_1.$
Applying the maximum principle again we get
\begin{lemma}{\textbf{[$C^1$- bounds on $\partial\Omega_1$]}}
\label{dp-c1outside-lem}
Let $\ue$ be the solution of \eqref{dp.1}, then on $\partial\Omega_1,$ $\ue$ satisfies
\[C\frac{\partial\psi}{\partial\nu}<\frac{\partial\ue}{\partial\nu}<\frac{\partial\psi}{\partial\nu}.\]
Here $\nu$ is the inward unit normal of $\partial\Omega^\e,$ i.e., $\nu$ points into $\Omega_1\setminus\ol B_\e(0).$
\end{lemma}
Finally, we will give a $C^1$ bound for $\ue$ in $\Omega^\e.$
\begin{lemma}{\textbf{[$C^1$- bound in $\Omega^\e$]}}
\label{dp-c1global-lem}
Let $\ue$ be the solution of \eqref{dp.1}, then
\[\max\limits_{\Omega^\e}|D\ue|=\max\limits_{\partial\Omega^\e}|D\ue|.\]
\end{lemma}
\begin{proof}
Let $V=|D\ue|^2,$ a direct calculation yields
\[
\begin{aligned}
\s_k^{ij}V_{ij}&=\s_k^{ij}(2\ue_l\ue_{lij}+2\ue_{li}\ue_{lj})\\
&=2\s_k^{ij}\ue_{li}\ue_{lj}>0,
\end{aligned}
\]
where $\s_k^{ij}=\frac{\partial\s_k}{\partial u_{ij}}.$ Lemma \ref{dp-c1global-lem} then follows from the maximum principle.
\end{proof}

We also need the following interior gradient estimates, which is proved in Theorem 3.1 of \cite{Tru97}.
\begin{lemma}{\textbf{[$C^1$- interior estimates in $\Omega^\e$]}}
\label{dp-c1interior-lem}
Let $\ue$ be the solution of \eqref{dp.1}, then for any ball $B=B_r(y)\subset\Omega^\e$ we have the estimate
\[|D\ue(y)|\leq\frac{C}{r}\text{osc}_B\ue,\]
where $C$ is a constant depending on $k$ and $n.$
\end{lemma}
Now, denote $U_\delta:=\{x\in\Omega^\e: \text{dist}(x, \partial\Omega_1)<\delta\}$, then by Lemma \ref{dp-c0-lem}, Lemma \ref{dp-c1outside-lem},
Lemma \ref{dp-c1global-lem}, and Lemma \ref{dp-c1interior-lem} we conclude
\begin{lemma}{\textbf{[$C^1$-estimates near $\partial\Omega_1$]}}
\label{dp-c1Udelta-lem}
Let $\ue$ be the solution of \eqref{dp.1}, then in $U_\delta$ we have the estimate
\[|D\ue|\leq C,\]
where $C=C(\delta)$ is a constant depending on $\delta$ but independent of $\e.$
\end{lemma}

\subsection{$C^2$ estimates of $\ue$}
\label{dpsub3}
Let $p\in\partial B_\e$ be an arbitrary point on $\partial B_\e.$ Without loss of generality, we may choose local coordinates
$\{\tilde x_1, \cdots, \tilde x_n\}$ in a neighborhood of $p$ such that $p$ is the origin. Let $\tilde x_n$ axis be the inward normal of
$\partial B_\e$ (pointing into $\Omega^\e$) then following the argument in \cite{CNS3} (see page 271, equation (1.8)) we obtain
\be
\label{dp3.1}
\ue_{\alpha\beta}(p)=\ue_\nu\e^{-1}\delta_{\alpha\beta}\,\,\mbox{for $p\in\partial B_\e$,}
\ee
where $u^\e_\nu:=\frac{\partial u^\e}{\partial\tilde x_n}.$
In view of Remark \ref{dprmk1} we have
\begin{remark}
\label{dprmk2}
On $\partial B_\e$ when $\frac{n}{k}>2,$ $|\ue_{\alpha\beta}|<C_0\lt(\frac{n}{k}-2\rt)\e^{-2};$
when $\frac{n}{k}=2,$ $|\ue_{\alpha\beta}|<C_0\e^{-2}|\log\e|^{-1};$ and when $\frac{n}{k}<2,$ $|\ue_{\alpha\beta}|<C_0\lt(2-\frac{n}{k}\rt)\e^{-n/k}.$
Here $C_0>0$ is a positive constant only depending on $\Omega_1.$
\end{remark}
In the following, we will establish the $C^2$-boundary estimates in the tangential normal directions and in the double normal directions.

We start with estimating the $C^2$ estimates in the tangential normal directions on $\partial B_\e(0).$
We denote the angular derivative $x_k\frac{\partial}{\partial x_l}-x_l\frac{\partial}{\partial x_k}$
by $\partial_{k, l}.$ For our convenience, we let $\partial:= x_\alpha\frac{\partial}{\partial x_n}-x_n\frac{\partial}{\partial x_\alpha}.$
Here, we assume $x_n$ to be the radial direction and $1\leq\alpha\leq n-1$ is a fixed integer.

\begin{lemma}
\label{dpsub3-lem1}
Let $\ue$ be the solution of \eqref{dp.1}, then we have
\[|\partial \ue|\leq C\,\,\mbox{in $\bar\Omega^\e,$}\]
where $C>0$ is a constant independent of $\e.$
\end{lemma}
\begin{proof}
It is clear that on $\partial B_\e$ we have $\partial \ue=0.$ Moreover, by the virtue of
Lemma \ref{dp-c1outside-lem} we get on $\partial\Omega_1,$ $|\partial \ue|\leq C$ for some $C>0$ only depending on $\Omega_1.$
By \cite{CNS1} we know that
\[\s_k^{ij}(\partial \ue)_{ij}=\partial\s_k(D^2\ue)=0.\]
The Lemma follows from the maximum principle.
\end{proof}

In view of Subsection \ref{dpsub1}, we know that by our choice of $M,$
$$\ul u^\e=\frac{|x|^2}{2C_n^k}-\frac{\e^2}{2C_n^k}-M$$
is a subsolution of \eqref{dp.1}. Let $h=\ue-\ul u^\e$ and $\mathcal L:=\sigma_k^{ij}\partial_{ij}.$
It's clear that $h=0$ on $\partial B_\e,$ $h>c_0>0$ on $\partial\Omega_1$ for some $c_0$ independent of $\e,$
and $\mathcal Lh\leq0$ in $\Omega^\e.$ Combining with the results of Lemma \ref{dpsub3-lem1} and the standard maximum principle, we conclude that there exists a positive constant
A such that $Ah>|\partial \ue|$ in $\Omega^\e.$ Here $A>0$ is a constant independent of $\e.$ Therefore, for any $p\in\partial B_\e,$ we can rotate
$\{x_1, \cdots, x_n\}$ such that $p=(0, \cdots, 0, \e).$ Since $Ah>|\partial u|$ in $\Omega^\e$ and $Ah=|\partial \ue|=0$ on $\partial B_\e$ we obtain
\[\pm(\partial \ue)_n<Ah_n=A\lt[\ue_n-\frac{2\e}{(C_n^k)^{1/k}}\rt]<A_1\ue_n.\]
Here, by our choice of the orientation, $x_n$ points into $\Omega^\e.$
We conclude
\begin{lemma}
\label{dp-c2mix-inside-lem}(\textbf{$C^2$ bound on $\partial B_\e$ in mixed directions})
Let $\ue$ be the solution of \eqref{dp.1}, then on $\partial B_\e,$ we have
\[|(\ue)_{\tau\nu}|\leq \frac{C}{\e}\ue_\nu\,\,\mbox{on $\partial{B_\e},$}\]
where $\tau$ is an arbitrary unit tangential vector of $\partial B_\e,$ $\nu$ is the inward unit normal of $\partial B_\e$ (pointing into $\Omega^\e$), and $C>0$ is a constant independent of $\e.$
\end{lemma}

In the following, we will derive the $C^2$ bound of $\ue$ on $\partial B_\e$ in the double normal directions.
For any $p\in\partial B_\e,$ let $\{\tau_1, \cdots, \tau_{n-1}\}$ be the orthonormal frame of the tangent hyperplane of $\partial B_\e$ at $p,$
and let $\tau_n$ be the inward unit normal of $\partial B_\e$ at $p.$ Then at $p$ we have
\be\label{dp3.3}
\sigma_{k-1}(\bar D^2\ue)\ue_{nn}+\sigma_{k}(\bar D^2\ue)
-\sum\limits_{\beta=1}^{n-1}\sigma_{k-2}(\bar D^2\ue|\ue_{\beta\beta})(\ue_{\beta n})^2=1,
\ee
where $\bar D^2\ue:=(\ue_{\alpha\beta})_{1\leq\alpha, \beta\leq n-1}.$
Plugging \eqref{dp3.1} into \eqref{dp3.3} we get
\[C_{n-1}^{k-1}\ue_{nn}+C_{n-1}^k\ue_n\e^{-1}=\e^{k-1}(\ue_n)^{1-k}+C_{n-2}^{k-2}\e(\ue_n)^{-1}\sum\limits_{\beta=1}^{n-1}(\ue_{\beta n})^2.\]
In view of Lemma \ref{dp-c1inside-lem} and Lemma \ref{dp-c2mix-inside-lem} we obtain
\begin{lemma}
\label{dp-c2normal-inside-lem}(\textbf{$C^2$ bound on $\partial B_\e$ in double normal directions})
Let $\ue$ be the solution of \eqref{dp.1}, then on $\partial B_\e,$ we have
\[-C_1\ue_{\nu}\e^{-1}\leq(\ue)_{\nu\nu}\leq C_2\ue_\nu\e^{-1}\,\,\mbox{on $\partial{B_\e},$}\]
where $C_1, C_2>0$ are positive constants independent of $\e$ and $\nu$ is the inward unit normal of $\partial B_\e$ (pointing in to $\Omega^\e$).
\end{lemma}
Combining Remark \ref{dprmk1}, Remark \ref{dprmk2}, Lemma \ref{dp-c2mix-inside-lem}, and Lemma \ref{dp-c2normal-inside-lem} we conclude
\begin{lemma}
\label{dp-inside-c2-lem}
Let $\ue$ be the solution of \eqref{dp.1}. Then on $\partial B_\e$ we have,
when $\frac{n}{k}>2,$ $|D^2\ue|<C_1\e^{-2};$
when $\frac{n}{k}=2,$ $|D^2 \ue|<C_2\e^{-2}|\log\e|^{-1};$ and when $\frac{n}{k}<2,$ $|D^2\ue|<C_3\e^{-n/k}.$
Here $C_1, C_2,$ and $C_3>0$ are positive constants independent of $\e.$
\end{lemma}

Next, we will establish the $C^2$ estimates of $\ue$ on $\partial\Omega_1.$ Let $p\in\partial\Omega_1$ be an arbitrary point on $\partial\Omega_1.$
Without loss of generality, we may choose local coordinates $\{\tilde x_1, \cdots, \tilde x_n\}$ at $p$ such that $\tilde x_n$ axis is the inward normal of
$\partial\Omega_1.$ Then the boundary near $p$ can be written as
\[\tilde x_n=\rho(\tilde x')=\frac{1}{2}\sum\limits_{\alpha=1}^{n-1}\kappa_\alpha\tilde x_\alpha^2+O(|\tilde x'|^3),\]
where $\kappa_1, \cdots, \kappa_{n-1}$ are the principal curvatures of $\partial\Omega_1$ at $p$
and $\tilde x'=(\tilde x_1, \cdots, \tilde x_{n-1}).$
Let $T=\frac{\partial}{\partial\tilde x_{\alpha}}+\kappa_\alpha\left(\tilde x_\alpha\frac{\partial}{\partial \tilde x_n}-
\tilde x_n\frac{\partial}{\partial\tilde x_\alpha}\right),$
and denote $\tilde B_{\delta_0}=B_{\delta_0}(p)\cap\Omega^\e.$ Let $\mathcal L:=\sigma_k^{ij}\partial_{ij},$ then we have
\[
\begin{aligned}
\mathcal LTu&=0\,\,\mbox{in $\tilde B_{\delta_0}$}\\
Tu&=O(|\tilde x'|^2)\,\,\mbox{on $\partial\tilde B_{\delta_0}\cap\partial\Omega_1.$}
\end{aligned}
\]
In view of Lemma \ref{dp-c1Udelta-lem}, we also know that $|Tu|\leq C$ on $\partial\tilde B_{\delta_0}\setminus\partial\Omega_1$
for some $C>0$ independent of $\epsilon$ and $\delta_0.$ Here we always assume $\delta_0<\delta$ and $\delta>0$ is the constant in Lemma \ref{dp-c1Udelta-lem}.

Let $\underline u=C\psi$ be a subsolution of \eqref{dp.1} for some $C>2$. Since $\underline u$ is $k$-convex, it is easy to see that there exists
$\theta>0$ such that
\[\lambda[D^2(\underline u-\theta|\tilde x|^2)]\in\Gamma_k\]
and
\[\sigma_k[D^2(\underline u-\theta|\tilde x|^2)]>(3/2)^k\,\,\mbox{in $\tilde B_{\delta_0}.$}\]
Here $\theta$ is a small constant independent of $\epsilon$ and $\delta_0.$ Consider
$h=u-\underline u+\theta|\tilde x|^2,$ we get $h\geq\theta\delta_0^2$ on $\partial\tilde B_{\delta_0}\setminus\partial\Omega_1$
and $h\geq 0$ on $\partial\tilde B_{\delta_0}\cap\partial\Omega_1$ with $h(p)=0.$ Moreover, by the concavity of
$\sigma_k^{1/k}$ we have $\mathcal Lh\leq-\frac{1}{2}k.$ In view of the standard maximum principle we conclude
$\pm Tu<Ah$ for some $A>0$ independent of $\epsilon$ (depending on $\delta_0$ though). Therefore, we have
\[|u_{\alpha n}(p)|<Ah_n(p)<A_1.\]
Here $A_1$ is a constant independent of $\epsilon.$

Following a similar argument as Lemma \ref{dp-c2mix-inside-lem}, we obtain $|u_{nn}|<C$ on $\partial\Omega_1$ for some $C>0$ independent of $\epsilon.$ We conclude
\begin{lemma}(\textbf{$C^2$ boundary estimates on $\partial\Omega_1$})
\label{c2-outsidebdry-lem}
Let $\ue$ be the solution of \eqref{dp.1}, then on $\partial\Omega_1,$ we have
\[|D^2\ue|<C\,\,\mbox{on $\partial{\Omega_1},$}\]
where $C>0$ is a positive constant independent of $\e$.
\end{lemma}

\begin{lemma}(\textbf{$C^2$ global estimates})
\label{c2-global-lem}
Let $\ue$ be the solution of \eqref{dp.1},then we have
\[|D^2\ue|<C(1+\sup\limits_{\partial\Omega^\e}|D^2\ue|)\]
where $C>0$ is a positive constant independent of $\e$.
\end{lemma}
\begin{proof}
In the following we will drop the superscript $\e$ and write $u$ instead of $\ue.$ Since $u$ is a solution of \eqref{dp.1}, $u$ satisfies $\sigma_k^{1/k}(D^2 u)=1.$ Differentiating this equality twice we get
\[F^{ij}(\Delta u)_{ij}+F^{pq, rs}u_{pql}u_{rsl}=0,\]
where $F^{ij}=\frac{\partial\sigma_k^{1/k}}{\partial u_{ij}}$ and $F^{pq, rs}=\frac{\partial^2\sigma_k^{1/k}}{\partial u_{pq}\partial u_{rs}}.$
In view of the concavity of $\sigma_k^{1/k}$ we get $F^{ij}(\Delta u)_{ij}\geq 0.$ Therefore,
\[\Delta u<C\left(1+\sup\limits_{\partial\Omega^\e}\Delta u\right),\]
which implies
\[|D^2 u|<C(1+\sup\limits_{\partial\Omega^\e}|D^2u|).\]
\end{proof}

Combining Lemmas \ref{dp-c0-lem}, \ref{dp-c1global-lem}, and \ref{c2-global-lem} we conclude
\begin{theorem}\label{dp-existence}
For any $\e>0,$ there exists a unique $k$-convex solution $u\in C^\infty(\Omega^\e)$ satisfying \eqref{dp.1} with
\[\|u\|_{C^2}<C,\]
where $C=C(\e)>0$ depends on $\e.$
\end{theorem}
Note that similar results and techniques of section \ref{dirichlet problem} also appeared in Ma-Zhang \cite{MZ22} and Xiao \cite{Xiao22}.

\bigskip

\section{Hessian measures}
\label{hm}

Let $\{u^{\e_m}\}$ be a sequence of solutions of \eqref{dp.1} with $\e_m\goto 0$ as $m\goto\infty.$ In this section, we will show that $\{u^{\e_m}\}$ converges locally in measure to a function $v$, and $v$ satisfies
$F_k[v]=1$ in the viscosity sense.

The following definition of $k$-convex function is an extension of the Definition \ref{def1.0}.
\begin{defin}(See \cite{TW99})
An upper semi-continuous function $u:\Omega\rightarrow [-\infty, \infty)$ is called \textbf{$k$-convex} in $\Omega$
 if $F_k[q]:=\sigma_k(D^2 q)\geq 0$ for all quadratic polynomials $q$ for which the difference $u-q$ has a finite local maximum in $\Omega.$
 We shall also call a $k$-convex function \textbf{proper} if it doesn't assume the value $-\infty$ identically on any component of $\Omega.$ We denote the class of
 proper $k$-convex functions in $\Omega$ by $\Phi^k(\Omega).$
\end{defin}

Since we will need to use mollifier to smooth our functions, we want to extend the domain of definition a little bit. Recall that $\partial\Omega_1$ is smooth and
$u^\epsilon\in C^\infty(\bar\Omega^{\epsilon}),$ where $\ue$ is the solution of \eqref{dp.1}. We can extend $u^\e$ to the other side of $\partial\Omega_1$ by Taylor's expansion in the normal bundle. We will still denote this expansion by $u^\e,$ and $u^\e$ is defined on $\Omega_1^\delta\setminus B_\e(0),$ where
\[\Omega_1^\delta:=\{x\in\mathbb R^n\mid\text{dist}(x, \Omega_1)<\delta\}\]
for some fixed small $\delta>0.$ Moreover, $u^\e$ satisfies
\[
\begin{aligned}
\sigma_k(D^2 u^\e)&=1\,\,\mbox{in $\Omega_1\setminus\bar B_{\e}(0)$}\\
\sigma_k(D^2 u^\e)&>1/2\,\,\mbox{in $\Omega_1^{\delta}\setminus\bar B_{\e}(0)$}.
\end{aligned}
\]
We will denote
\[\tilde u^\e=\left\{
\begin{aligned}
u^\e\,\,&\mbox{in $\Omega^\delta_1\setminus B_\e(0)$}\\
-M\,\, &\mbox{in $\bar B_\e(0).$}
\end{aligned}
\right.
\]
\begin{lemma}
\label{hmlem1}
Let $\tilde u^\e$ be defined as above, then $\tilde u^\e\in\Phi^k(\Omega_1^\delta).$
\end{lemma}
\begin{proof}
It is clear that $\tilde u^\e$ is a continuous function in $\Omega^\delta_1.$ In the following we will show $\tilde u^\e$ is $k$-convex.

Let $q$ be any quadratic polynomials such that $\tilde u^\e-q$ has a local maximum at some point $x\in\Omega_1^\delta.$
When $x\in\Omega_1^\delta\setminus\partial B_\e(0),$ since $\tilde u^\e$ is smooth in a small neighborhood of $x$ and
$\sigma_k(D^2\tilde u^\e)\geq 0$ in this neighborhood, we have $\sigma_k(D^2 q)\geq 0.$

Now, consider the case when $x\in\partial B_\e(0).$ Let $\ul u^\e$ be the subsolution of \eqref{dp.1} defined in \eqref{lower barrier}. Since $\underline u^\e=\tilde u^\e$ on $\partial B_\e$ and $\underline u^\e<\tilde u^\e$
in $\Omega_1.$ It is clear that $\underline u^\e-q$ also achieves a local maximum at $x$ and we have $\sigma_k(D^2 q)\geq 1.$
\end{proof}

In the following we will look at the mollification of $\tilde u^\e.$ Let $\rho\in C_0^\infty(\mathbb R^n)$ be a spherically symmetric mollifier
satisfying $\rho(x)>0$ for $|x|<1,$ $\rho(x)=0$ for $|x|\geq 1,$ and $\int \rho=1.$ The mollification, $\tilde u^\e_h$ is defined by
\[\tilde u^\e_h=h^{-n}\int_{\mathbb R^n}\rho\left(\frac{x-y}{h}\right)\tilde u^\e(y)dy\]
for $0<h<\text{dist}(x, \partial\Omega_1^\delta).$

From the definition of mollification and Lemma 2.3 of \cite{TW99} we obtain
\begin{lemma}
\label{hmlem2}
$\tilde u^\e_h\in C^\infty(\Omega')\cap\Phi^k(\Omega')$ for any $\Omega'\subset\Omega_1^\delta$ satisfying $\text{dist}(\Omega', \partial\Omega_1^\delta)\geq h.$
Moreover, as $h\rightarrow 0,$ the sequence $\tilde u^\e_h\rightarrow \tilde u^\e.$
\end{lemma}

Notice that $\tilde u^{\e_1}(x)\geq \tilde u^{\e_2}(x)$ in $\Omega_1$ whenever $\e_1<\e_2.$ In conjunction with earlier $C^0$ and $C^1$ interior estimates, see Lemma \ref{dp-c0-lem} and Lemma \ref{dp-c1interior-lem}, we have
$\tilde u^\e\longrightarrow v$ in $C_{loc}^1\left(\Omega_1\setminus \{0\}\right).$ Note that since $\tilde u^\e(0)=-M$ for any $\e>0$, we get $v(0)=-M,$ thus $v$ may not be an upper-semi continuous function.
We redefine $v(0)=\limsup\limits_{x\rightarrow 0}v(x),$ and in the following, all $v(x)$ refers to this redefined $v(x).$

\begin{lemma}
\label{hmlem3}
Let $v(x)$ be defined as above, then $v(x)\in\Phi^k(\Omega_1).$
\end{lemma}
\begin{proof}
We will prove by contradiction. If $v(x)\notin\Phi^k(\Omega_1)$ then there exists a quadratic polynomial such that
$v(x_0)-q(x_0)=0,$ $v(x)-q(x)\leq 0$ for all $x\in B_\delta(x_0)\Subset\Omega_1,$ and $F_k[q]<0.$ Without loss of generality, we may also assume
$v(x)-q(x)<c_0<0$ on $\partial B_\delta(x_0).$ If not, we will consider $\hat v(x):=v(x)-\beta|x-x_0|^4$ instead, and correspondingly, we will replace
$\tilde u^\e$ by $\hat{\tilde{u}}^\e(x):=\tilde u^\e-\beta|x-x_0|^4$. Here, $\beta>0$ is a very small constant.

Case 1. When $x_0\neq 0,$ we may assume $0\notin B_\delta(x_0).$ Moreover, when $\e>0$ small, we also have $\bar B_\delta(x_0)\cap \bar B_\e(0)=\emptyset.$ Since $\tilde u^\e(x)\rightarrow v(x)$ uniformly in
$\bar B_\delta(x_0),$ we get for any $\eta>0$ there exists $\e_\eta>0$ such that when $\e<\e_\eta$ we have $\left|\tilde u^\e(x)-v(x)\right|<\eta$ for all $x\in \bar B_\delta(x_0).$
Therefore, we get $\tilde u^\e(x_0)-q(x_0)>-\eta$ and
$$\tilde u^\e(x)-q(x)<v(x)-q(x)<c_0<0\,\,\mbox{on $\partial B_\delta(x_0).$}$$
Here, the first inequality comes from as $\e\searrow 0,$ $\tilde u^\e\nearrow v(x).$
We can see that when $\eta<|c_0|,$ $\tilde u^\e(x)-q(x)$ achieves its local maximum in $B_\delta(x_0).$ Recall that $\tilde u^\e$ satisfies $F_k[\tilde u^\e]=1$ in $B_\delta(x_0),$ we have $F_k[q]\geq 1.$ This leads to a contradiction. When we consider $\hat{\tilde{u}}^\e(x)$, for $\beta>0$ small we have $F_k[\hat{\tilde{u}}^\e]\geq 1/2$ in $B_\delta(x_0),$ thus $F_k[q]\geq 1/2.$

Case 2. When $x_0=0,$ we recall that $v(0)=\limsup\limits_{x\rightarrow 0} v(x).$ For any $\eta>0$ small, there exists
$\{x_n\}\subset B_\delta(0)\setminus\{0\}$ and $x_n\rightarrow 0$ such that when $n>N,$ $|v(x_n)-v(0)|<\eta/3.$ We may also assume
when $n>N,$ $|q(x_n)-q(0)|<\eta/3.$ Moreover, same as in case 1 we have,
$$\tilde u^\e(x)-q(x)<v(x)-q(x)<c_0<0\,\,\mbox{on $\partial B_\delta(0).$}$$
Now, for any $n>N,$ we fix $x_n,$ then there exists $\e_1=\e_1(\eta, x_n)>0$ such that when $\e<\e_1$ we have
$|\tilde u^\e(x_{n})-v(x_{n})|<\eta/3.$
Therefore, for this $x_n\in B_\delta(0),$ when $\e>0$ small enough we get
\[|\tilde u^\e(x_n)-q(x_n)|<|\tilde u^\e(x_n)-v(x_n)|+|v(x_n)-v(0)|+|q(0)-q(x_n)|<\eta.\]
Choosing $\eta<|c_0|$ yields $\tilde u^\e-q(x)$ achieves its local maximum in $B_\delta(0).$ By virtue of Lemma \ref{hmlem1} we obtain $F_k[q]\geq 0,$ which leads to a contradiction.
Since in this case, when replacing $\tilde u^\e$ by $\hat{\tilde{u}}^\e$ the argument in the last step is not so straightforward, we will include the details below.

We will assume $\max_{\bar B_\delta(0)}\hat{\tilde{u}}^\e-q(x)$ is achieved at $x^\e.$

Subcase 1. There exists $\e_2>0$ such that for all $\e<\e_2,$ $|x^\e|>\eta_1>0.$ In this case we have when $\e>0$ small enough,
at $x^\e,$ $F_k[\hat{\tilde{u}}^\e]\geq 1/2.$ This implies $F_k[q]\geq 1/2,$ which leads to a contradiction.

Subcase 2. There exists a sequence $\{\e_i\}\goto 0$ such that $|x^{\e_i}|\goto 0.$
In the following, without causing confusions, we will denote the subsequence of $\{x^{\e_i}\}$ by $\{x^{\e_i}\}$ as well.

When $\{x^{\e_i}\}\subset\Omega_1\setminus\bar B_{\e_i}(0),$ we have $F_k[\hat{\tilde{u}}^{\e_i}]\geq 1/2$ at $x^{\e_i}.$ This gives
$F_k[q]\geq 1/2,$ which leads to a contradiction.

When $\{x^{\e_i}\}\subset\partial B_{\e_i}(0),$ we will consider $\underline u^{\e_i}-\beta|x-x_0|^4$ instead of $\hat{\tilde{u}}^\e.$
It's clear that $\underline u^{\e_i}-\beta|x-x_0|^4-q(x)$ achieves a local maximum at $x^{\e_i}.$ Therefore, at $x^{\e_i}$ we again have
$F_k[q]\geq 1/2,$ which leads to a contradiction.

When $\{x^{\e_i}\}\subset B_{\e_i}(0),$ we know that $\hat{\tilde{u}}^\e=-M-\beta|x-x_0|^4.$
Let $\lambda:=\lambda[D^2q]=(\lambda_1, \cdots, \lambda_n)$ be the eigenvalue vectors of $\{D^2 q\}$
and let $\lambda_{\min}=\min\{\lambda_1, \cdots, \lambda_n\}.$ Then at $x^{\e_i}$ we have
$\lambda_{\min}\geq -12\beta|x^{\e_i}|^2.$ Therefore, $F_k[q]\geq-C|x^{\e_i}|^2$ for some positive constant $C$ that is independent of $\e_i.$
Let $\e_i\goto 0$ we obtain $F_k[q]\geq 0,$ which leads to a contradiction.
\end{proof}

Now, let $\{u^{\e_m}\}$ be any sequence of solutions of \eqref{dp.1} with $\e_m\goto 0$ as $m\goto\infty.$ We denote $u^m:=\tilde u^{\e_m}_{h_m}$ to be the mollification of $\tilde u^{\e_m}.$ Moreover, as $m\rightarrow \infty$ we have
$\e_m, h_m\rightarrow 0.$ Here, we always assume $h_m\ll\e_m.$ Combining Lemma \ref{dp-c0-lem}, \ref{dp-c1interior-lem}, and \ref{hmlem3} with the fact that $v(0)$ is finite,
it is clear that
$\{u^m\}\subset\Phi^k(\Omega_1)\cap C^\infty(\Omega_1)$ converges in $L_{loc}^1(\Omega_1)\cap C^1_{loc}\left(\Omega_1\setminus\{0\}\right)$ to $v\in\Phi^k(\Omega_1).$
By Theorem 1.1 of \cite{TW99} we know $\mu_k[u^m]\rightharpoonup\mu_k[v].$ In view of the Portmanteau theorem this is equivalent to say, for any $B=B_r(x)\Subset\Omega_1$ we have
\[\mu_k[v](B)\leq\liminf\limits_{m\rightarrow\infty}\mu_k[u^m](B)\]
and
\[\mu_k[v](\bar B_{\sigma r})\geq\limsup\limits_{m\rightarrow\infty}\mu_k[u^m](\bar B_{\sigma r})\]
for any $\sigma\in (0, 1).$

\begin{lemma}
\label{hmlem4}
For $k\leq n/2,$ we have $\mu_k[v]=\nu_E$ in $\Omega_1,$ where $\nu_E$ is the standard measure on $\mathbb R^n.$
\end{lemma}
\begin{proof}
We only need to prove for any $B=B_r(x)\Subset\Omega_1$ we have $\mu_k[v](B)=\nu_E(B).$
Let $A_m=B_{3\e_m/2}(0)\setminus \bar B_{\e_m/2}(0)$ then
\[B=\{B\cap(\Omega_1\setminus B_{3\e_m/2})\}\cup\{B\cap A_m\}\cup\{B\cap\bar B_{\e_m/2}\}:=I_1\cup I_2\cup I_3\]
Now, for any $\eta>0,$ by the well known properties of mollifications we know there exists $h_{\eta}>0$ such that
when $0<h_m<h_\eta,$
\[
\begin{aligned}
|u^m-\tilde u^{\e_m}|&<\eta\,\,\mbox{in $\{\bar\Omega_1\setminus B_{4\e_m/3}\}\cup\bar B_{2\e_m/3}$}\\
|Du^m-D\tilde u^{\e_m}|&<\eta\,\,\mbox{in $\{\bar\Omega_1\setminus B_{4\e_m/3}\}\cup\bar B_{2\e_m/3}$}\\
|D^2u^m-D^2\tilde u^{\e_m}|&<\eta\,\,\mbox{in $\{\bar\Omega_1\setminus B_{4\e_m/3}\}\cup\bar B_{2\e_m/3}$}.
\end{aligned}
\]
Therefore, for $n/k>2$ we have
\be\label{h,.1}
\begin{aligned}
\mu_k[u^m](B)&=\int_{I_1}F_k[u^m]dx+\int_{I_2}F_k[u^m]dx+\int_{I_3}F_k[u^m]dx\\
&\leq\int_{I_1}\sigma_k(D^2\tilde u^{\e_m}+\eta I)+\int_{\partial B_{3\e_m/2}(0)}\sigma_k^{ij}u_i^m\gamma_jd\sigma\\
&-\int_{\partial B_{\e_m/2}}\sigma_k^{ij}u_i^m\gamma_jd\sigma+\int_{I_3}\sigma_k(D^2\tilde u^{\e_m}+\eta I)dx\\
&\leq \nu_E(I_1)+C_1\sum\limits_{j=0}^{k-1}\eta^{k-j}\int_{I_1}\sigma_j[\tilde u^{\e_m}]\\
&+C_2\e_m^{n-1}\frac{1}{|\e_m^2-\eta|^{k-1}|\e_m-\eta|}+C_3\eta^k\e_m^n.
\end{aligned}
\ee
Here, $\gamma$ is the unit exterior normal to $\partial A_m.$ Moreover, we have used the divergence theorem to derive the first inequality, and we have used Remark \ref{dprmk1},
Lemma \ref{dp-inside-c2-lem}, and Lemma \ref{c2-global-lem} to derive the second inequality.
We also note that $C_1, C_2,$ and $C_3$ are positive constant only depending on $n, k,$ and $\Omega_1.$ Now let $\eta\leq\e_m^{2k}$ then we have
\[\mu_k[u^m](B)\leq\nu_E(B)+C_4\e_m^{n-2k}+C_5\eta\e_m^{-2(k-1)}\] for some $C_4, C_5>0$ are independent of the choice of $B.$
Therefore, we obtain $\mu_k[v](B)\leq\nu_E(B).$
Similarly, we can show that for any $\sigma\in (0, 1)$
\[\mu_k[v](\bar B_{\sigma r})\geq\nu_E(\bar B_{\sigma r}).\]
Let $\sigma\rightarrow 1$ we obtain $\mu_k[v](B)\geq\nu_E(B).$ When $n/k=2$ the proof is similar. This completes the proof of this Lemma.
\end{proof}

\bigskip

\section{Proof of Theorem \ref{thm1}.}
\label{pf}
\begin{proof}
The strategy of the proof follows \cite{HNS16}. We have proved the existence of a solution $u^\e$
of equation \eqref{dp.1}. The uniquesness of this solution follows directly from the maximum principle. We want to show that $u^\e$ has some non-convex sublevel sets for some $\e>0$ small enough.
In Section \ref{hm} we have shown that $\tilde u^\e\rightarrow v$ in $L^1_{loc}(\Omega_1)\cap C_{loc}^1(\Omega_1\setminus\{x_0\})$ and $v$ satisfies
\[
\begin{aligned}
\mu_k[v]&=1\,\,\mbox{in $\Omega_1$}\\
v&=0\,\,\mbox{on $\partial\Omega_1$}.
\end{aligned}
\]
By the uniqueness of weak solutions, which is proved in Lemma 4.2 of \cite{TW02}, we get $v=\psi$ in $\bar\Omega_1.$
Here, $\psi$ has been defined in Subsection \ref{dpsub1}.
Now, let us assume by contradiction that for each $\e>0$ all sublevel sets of $\tilde u^\e$ are convex. Then there exists a sequence
$\{\e_n\}\subset (0, \e_0)$ such that $\e_n\rightarrow 0$ as $n\rightarrow\infty,$ and for each $n\in N$ the sublevel sets of
$\tilde u^{\e_n}$ are all convex. Fix a point $y\in\Omega_1$ such that
\[\psi(y)=\inf\limits_{\bar\Omega_1}\psi=-M_0,\] we may choose $x_0\neq y$ such that $\psi(x_0)>-M_0.$ In the following we consider
\eqref{dp.1} in the domain $\Omega_1\setminus B_\e(x_0).$
Let $(x_n)_{n\in\mathbb N}$ be any sequence of points in $\bar\Omega_1$ such that $x_n\in B_{\e_n}(x_0)$ for all $n\in\mathbb N.$

Since $y\neq x_0,$ we know that $\tilde u^{\e_n}(y)\rightarrow \psi(y)$ as $n\rightarrow\infty.$ Therefore there exists $n_0\in\mathbb N$ such that for
$n\geq n_0$ we have $\tilde u^{\e_n}(y)\leq\psi(y)+\eta.$
Here $\eta>0$ is an arbitrary positive real number. Moreover, $\tilde u^{\e_n}(x_n)=-M<-M_0=\psi(y).$
By our assumption that the sublevel sets of $\tilde u^{\e_n}$ in $\bar\Omega_1$ are convex, we get
\[\tilde u^{\e_n}(x)\leq\psi(y)+\eta\,\,\mbox{for all $x\in [x_n, y].$}\]
By our choice of $x_n$ we know $x_n\goto x_0$ as $n\goto\infty.$ One infers that
$\psi(x)\leq\psi(y)+\eta$ for all $x\in (x_0, y].$ Then by the continuity of $\psi$ we get
$\psi(x_0)\leq\psi(y)+\eta.$ Since $\eta>0$ is arbitrary we conclude $\psi(x_0)\leq -M_0,$
which is ruled out by the choice of $x_0.$ Thus the Theorem is proved.
\end{proof}

\end{document}